\documentclass[12pt]{article}
\usepackage{amsmath, amssymb, amsfonts, latexsym, amsthm, amscd, epsfig, graphics, subfig}
\usepackage{hyperref}
\usepackage{lineno}

\newtheorem{theorem}{Theorem}[section]

\newtheorem{lemma}[theorem]{Lemma}

\newtheorem{corollary}[theorem]{Corollary}
\newtheorem{example}[theorem]{Example}
\newtheorem{remark}[theorem]{Remark}
\numberwithin{equation}{section}
\allowdisplaybreaks[4]

\newcommand*\patchAmsMathEnvironmentForLineno[1]{%
  \expandafter\let\csname old#1\expandafter\endcsname\csname #1\endcsname
  \expandafter\let\csname oldend#1\expandafter\endcsname\csname end#1\endcsname
  \renewenvironment{#1}%
     {\linenomath\csname old#1\endcsname}%
     {\csname oldend#1\endcsname\endlinenomath}}%
\newcommand*\patchBothAmsMathEnvironmentsForLineno[1]{%
  \patchAmsMathEnvironmentForLineno{#1}%
  \patchAmsMathEnvironmentForLineno{#1*}}%
\AtBeginDocument{%
\patchBothAmsMathEnvironmentsForLineno{equation}%
\patchBothAmsMathEnvironmentsForLineno{align}%
\patchBothAmsMathEnvironmentsForLineno{flalign}%
\patchBothAmsMathEnvironmentsForLineno{alignat}%
\patchBothAmsMathEnvironmentsForLineno{gather}%
\patchBothAmsMathEnvironmentsForLineno{multline}%
}

\title{\bf Whitney-Graustein Homotopy of Locally Convex Curves via a Curvature Flow}

\author{\bf Laiyuan Gao}
\date{\today}

\begin{document}
\maketitle

\noindent {\bf Abstract}
Let $X_0, \widetilde{X}$ be two smooth, closed and locally convex curves in the plane with same winding number.
A curvature flow with a nonlocal term is constructed to evolve $X_0$ into $\widetilde{X}$.
It is proved that this flow exits globally, preserves both the local convexity and the elastic energy of the evolving curve.
If the two curves have same elastic energy then the curvature flow deforms the evolving curve
into the target curve $\widetilde{X}$ as time tends to infinity.
\\\\


\baselineskip 15pt

\section{Introduction}
\setcounter{equation}{0}

In 1937, H. Whitney \cite{Whitney-1937} showed that two smooth and closed curves in the Euclidean plane may be smoothly deformed to each other
if and only if the two curves have same winding number. Whitney in his paper said that this result and its proof had been
suggested by W. C. Graustein, so this fact is called Whitney-Graustein Theorem now.

Since 1980, geometers have created different kinds of curvature flows to study the deformation of curves, surfaces and
higher dimensional manifolds.
The developments of these curvature flows play very important roles in geometry and topology.
Apart from those higher dimensional arts, there are some profound and influential results on the curvature flow of curves, such as
the curve shortening flow by Gage \cite{Gage-1983, Gage-1984}, Gage-Hamilton \cite{Gage-Hamilton-1986} and Grayson \cite{Grayson-1987, Grayson-1989},
the expanding flow by Chow-Liou-Tsai \cite{Chow-Liou-Tsai} and Tsai \cite{Tsai-2005}, the anisotropic flow by Chou-Zhu \cite{Chou-Zhu-1999-1, Chou-Zhu-1999-2}
and the applications of curve flows to classical geometry by Angenent \cite{Angenent-2005} and so on.
In this situation, S. T. Yau in 2007 \cite{Lin-Tsai-2009} asked that whether one can use a parabolic curvature flow method to
evolve one curve to another. An answer to Yau's question is a realization
of Whitney-Graustein differential homotopy for closed curves via a curvature flow.

In order to settle this problem, Lin and Tsai in their paper \cite{Lin-Tsai-2009} defined a new parabolic model to evolve one convex curve
to another. They showed that if two convex curves have same length then their flow can deform one curve into the other, provided that
the flow exists globally.
Later, Tsai \cite{Tsai-2018} found the blow-up phenomenon of this linear flow.
Following Gage \cite{Gage-1993} and Gage-Li \cite{Gage-Li-1994}, Pan and Yang \cite{Pan-Yang-2019}
in 2017 studied a nonlocal flow which evolves convex curves
into a given centrosymmetric convex one. In the same time, inspired by Lin-Tsai \cite{Lin-Tsai-2009} and
Chou-Zhu \cite{Chou-Zhu-1999-1, Chou-Zhu-1999-2}, Gao and Zhang \cite{Gao-Zhang-2019} generalized
Gage's area-preserving flow \cite{Gage-1986} and proved that the generalized flow exists globally and, up to a
rescaling, evolves one convex curve to another given one. So the convex case of Yau's above question
of evolving one curve to another has been solved.

In this paper, the author continues to study Yau's question for the case of locally convex curves.
If a $C^2$ and closed curve in the plane has positive (relative) curvature everywhere then it is called locally convex.
Apart from convex curves, there are uncountably many other locally convex curves.
This is a very special phenomenon in the planar geometry, because compact  and locally convex hypersurfaces in higher
dimensional Euclidean spaces are all convex ones (see Hadamard's theorem \cite{Hadamard-1898} or \cite{Hopf-1983}).
Due to this reason, the curvature flows of locally convex curves in the plane arose some particular interests in the past several years
(see Chen-Wang-Yang \cite{Chen-Wang-Yang-2017}, Wang-Li-Chao \cite{Wang-Li-Chao-2017}, Wang-Wo-Yang \cite{Wang-Wo-Yang-2018}).
Locally convex curves also play an important role in understanding the asymptotic behavior of the famous
curve shortening flow (see Abresch-Langer \cite{Abresch-Langer-1986}, Altschuler \cite{Altschuler-1991},
Angenent \cite{Angenent-1991} and Epstein-Gage \cite{Epstein-Gage-1987}) and its generalization
(see Andrews \cite{Andrews-2002}).

Let $X_0: [0, L_0] \rightarrow \mathbb{R}^2 (s \mapsto (x(s), y(s)))$ be a $C^2$ and closed curve in the plane,
where $s$ is the arc length parameter and $L_0$ is the length. Denote by $\{T, N\}$ the Frenet frame of this curve, i.e.,
for each $s$, the ordered pair $(T(s), N(s))$ determines a positive orientation of the plane.
Denote by $\kappa_0(s): = \langle \frac{d T(s)}{ds}, N(s)\rangle$ the curvature of the curve $X_0$.
The winding number of $X_0$ is defined by
\begin{eqnarray}\label{eq:1.1.202003}
m(X_0):= \frac{1}{2\pi} \int_0^{L_0} \kappa_0(s) ds.
\end{eqnarray}
If $X_0$ is locally convex then $\kappa_0$ is positive everywhere and $m(X_0)$ is a positive integer.
The elastic energy of the curve (see \cite{Gao-Wang-2008}, \cite{Singer-2008})
is defined by  $$E: = \int_0^{L_0} (\kappa_0(s))^2 ds.$$
The function $p_0(s):= -\langle X_0(s), N(s)\rangle$ is called the support function of the curve.
Let $\theta$ be the tangent angle, i.e., the angle from positive direction of $x$-axis to the unit tangent vector.
Since $\frac{d \theta}{d s} = \kappa_0(s)$ is positive for each $s$, $\theta$ can be used as a parameter of $X_0$.
And the curve $X_0: mS^1 \rightarrow \mathbb{E}^2 (\theta \mapsto (x(\theta), y(\theta)))$ is
a $C^{2}$ mapping from the $m$-fold circle to the plane.
If the locally convex curve $X_0$ is parametrised
by the tangent angle $\theta$ then the Frenet formula is as follows
\begin{eqnarray}\label{eq:1.2.202003}
\frac{d T}{d \theta}(\theta) = N(\theta), ~~\frac{d N}{d \theta}(\theta) = -T(\theta).
\end{eqnarray}
Furthermore, differentiating the support function gives us
\begin{eqnarray}\label{eq:1.3.202003}
\frac{d p_0}{d \theta}(\theta) = \langle X_0(\theta), T(\theta)\rangle,
~~\frac{d^2 p_0}{d \theta^2}(\theta) = \rho_0(\theta)-p_0(\theta),
\end{eqnarray}
where $\rho_0(\theta)= \frac{1}{\kappa_0(\theta)}$ is the radius of the curvature.

Let $X_0, \widetilde{X}: mS^1 \rightarrow \mathbb{E}^2$ be two smooth and locally convex curves in the plane with same winding number $m$.
Denote by
$$X:  mS^1 \times [0, \omega)  \rightarrow \mathbb{E}^2 ~~((\theta, t) \mapsto (x(\theta, t), y(\theta, t)))$$
a family of locally convex curves
with $X(\theta, 0) = X_0(\theta)$, where $\theta$ is the tangent angle. Since the locally convex curve
$\widetilde{X}$ has the same winding number with $X(\cdot, t)$, $\theta$ can also be used as a parameter for this curve.
Let $p(\theta, t)$ be the support function of the curve $X(\cdot, t)$ and let $\rho(\theta, t)$ be its radius of curvature.
Denote by $\widetilde{p}(\theta)$
and $\widetilde{\rho}(\theta)$ the support function and the radius of curvature of the curve $\widetilde{X}(\cdot)$, respectively.
In order to answer Yau's question for the case of locally convex curves, the next curvature flow is introduced:
\begin{eqnarray}\label{eq:1.4.202003}
\left\{\begin{array}{ll}
\frac{\partial X}{\partial t}(\theta, t) = \alpha(\theta, t) T(\theta, t)
         + \left[2p(\theta, t) -\rho(\theta, t) -2\widetilde{p}(\theta, t) +\widetilde{\rho}(\theta, t) + f(t)\right] N(\theta, t),
\\ ~~~~~~~~~~~~~~~~~~~~~(\theta, t) \in  mS^1\times [0, \omega), \\
X(\theta, 0) = X_0(\theta), ~~\theta \in  mS^1,
\end{array} \right.
\end{eqnarray}
where the coefficient of the tangent component is formulated as
\begin{eqnarray}\label{eq:1.5.202003}
\alpha(\theta, t) = -\frac{\partial}{\partial \theta}\big(2p(\theta, t) -\rho(\theta, t) -2\widetilde{p}(\theta, t)
  +\widetilde{\rho}(\theta, t)\big)
\end{eqnarray}
and the nonlocal term is given by
\begin{eqnarray}\label{eq:1.6.202003}
f(t)= \frac{1}{\int_{mS^1} \kappa^2 d\theta}
   \left[\int_{mS^1} \kappa^2 \left(\frac{\partial^2 \rho}{\partial \theta^2}-\frac{d^2 \widetilde{\rho}}{d \theta^2}\right) d\theta
- \int_{mS^1} \kappa^2 (\rho - \widetilde{\rho}) d\theta\right].
\end{eqnarray}

\begin{remark}\label{remk:1.1.202003}
The local convexity of the evolving curve is preserved under the flow (see Lemma \ref{lem:2.9.202003}),
so the tangent angle $\theta$ can be used as a parameter of the evolving curve and
the evolution equation (\ref{eq:1.4.202003}) is well defined for $t> 0$.
\end{remark}

\begin{remark}\label{remk:1.1.202003}
In contrast to the flow by Lin-Tsai \cite{Lin-Tsai-2009}, the support functions $p$ and $\widetilde{p}$ and a complicated nonlocal term $f(t)$
are used in the flow equation  (\ref{eq:1.4.202003}). Without these terms, one can not expect the global existence and the convergence
of the flow (\ref{eq:1.4.202003}) with a generic initial locally convex curve $X_0$.
See a blow-up example of the flow by Lin-Tsai \cite{Lin-Tsai-2009} in the paper \cite{Tsai-2018}.
\end{remark}

The evolution equation (\ref{eq:1.4.202003}) is a completely nonlinear parabolic system for the evolving curve
$X(\theta, t)=(x(\theta, t), y(\theta, t))$, where $(\theta, t) \in  mS^1\times [0, \omega)$.
The purpose of this paper is to partially answer Yau's question by understanding the asymptotic behavior
of the evolving curve $X(\cdot, t)$.
As an application in the field of topology, this curvature flow can be used to realize
Whitney-Graustein differential homotopy for locally convex curves. The main result of this paper is as follows.

\begin{theorem}\label{thm:1.2.202003}
Let $X_0$ and $\widetilde{X}$ be two smooth  and locally convex curves.
The flow (\ref{eq:1.4.202003}) with initial $X_0$ and target $\widetilde{X}$ exists globally,
preserves both the positivity of the curvature and the elastic energy of the evolving curve $X(\cdot, t)$.
If $X_0$ and $\widetilde{X}$ have same elastic energy, then $X(\cdot, t)$ converges, in the sense of $C^\infty$ metric,
to the target curve $\widetilde{X}$ as time $t\rightarrow +\infty$.
\end{theorem}

The key idea in the proof of Theorem \ref{thm:1.2.202003} is to reduce the nonlinear system to the evolution equation
of the radius of the curvature $\rho$ (see (\ref{eq:2.7.202003})). It is a half linear equation
with nonlinear part contained in the integral term $f(t)$.
The short time existence and the global existence of the flow (\ref{eq:1.4.202003}) are proved in Section 2.
The convergence of this nonlocal flow is proved in Section 3. An example is presented in Section 4.

\section{Existence}

\subsection{Short Time Existence}
In this subsection, we prove that the flow (\ref{eq:1.4.202003}) has a smooth solution on the domain
$mS^1 \times [0, t_0)$, where $t_0$ is a positive number.

Suppose there is a family of locally convex curves evolving under the flow (\ref{eq:1.4.202003}).
Denote $\beta (\theta, t) = 2p(\theta, t) -\rho(\theta, t) -2\widetilde{p}(\theta, t) +\rho(\theta, t) + f(t)$.
By direct calculations, one has the following evolution equations.

\begin{lemma}\label{lem:2.1.202003}
Applying the equations (1.14)-(1.17) in the book \cite{Chou-Zhu-1999-1}, one obtains
\begin{eqnarray}
&& \frac{\partial T}{\partial t} = \left(\alpha \kappa +\frac{\partial \beta}{\partial s} \right)N
= \left(\alpha +\frac{\partial \beta}{\partial \theta} \right)\kappa N,  \label{eq:2.1.202003}
\\
&& \frac{\partial N}{\partial t} = -\left(\alpha \kappa +\frac{\partial \beta}{\partial s} \right) T
= -\left(\alpha +\frac{\partial \beta}{\partial \theta} \right)\kappa,  \label{eq:2.2.202003}
\\
&& \frac{\partial \theta}{\partial t} =\alpha \kappa +\frac{\partial \beta}{\partial s}
= \left(\alpha +\frac{\partial \beta}{\partial \theta} \right)\kappa,
 \label{eq:2.3.202003} \\
&&\frac{\partial \kappa}{\partial t} = \kappa^2 \left(-\frac{\partial^2 \rho}{\partial \theta^2} + \frac{\partial^2 \widetilde{\rho}}{\partial \theta^2}
+ \rho - \widetilde{\rho} +f(t)\right).  \label{eq:2.4.202003}
\end{eqnarray}
\end{lemma}
By the choice of the tangent component (\ref{eq:1.5.202003}), we know $\alpha = -\frac{\partial \beta}{\partial \theta}$.
So both the Frenet frame $\{T, N\}$ and the tangent angle $\theta$ are independent of the time:
\begin{eqnarray}\label{eq:2.5.202003}
\frac{\partial T}{\partial t} \equiv 0, ~~\frac{\partial N}{\partial t} \equiv 0, ~~
\frac{\partial \theta}{\partial t} \equiv 0.
\end{eqnarray}
Using the above evolution equations, one can compute the evolution equation of the support function:
\begin{eqnarray*}
\frac{\partial p}{\partial t} = - \frac{\partial}{\partial t} \left\langle X, N \right\rangle
= - \left\langle  \frac{\partial X}{\partial t}, N\right\rangle = - \beta.
\end{eqnarray*}
So it follows from (\ref{eq:1.3.202003}) that
\begin{eqnarray}\label{eq:2.6.202003}
\frac{\partial p}{\partial t}
=\frac{\partial^2 p}{\partial \theta^2} (\theta, t) - p(\theta, t) +2\widetilde{p}(\theta, t) -\widetilde{\rho}(\theta, t) - f(t).
\end{eqnarray}

Using the equation (\ref{eq:2.4.202003}) and the fact that the radius of the curvature $\rho = \frac{1}{\kappa}$, we have
\begin{eqnarray}\label{eq:2.7.202003}
\frac{\partial \rho}{\partial t} = \frac{\partial^2 \rho}{\partial \theta^2} - \frac{\partial^2 \widetilde{\rho}}{\partial \theta^2}
- \rho + \widetilde{\rho} -f(t).
\end{eqnarray}
Since the function $\rho(\theta, t)$ determines the shape of the evolving curve, the flow (\ref{eq:1.4.202003})
can be reduced to the equation (\ref{eq:2.2.202003}) with initial $\rho(\theta, 0) = \rho_0(\theta)$ in some small time interval.
\begin{lemma}\label{lem:2.2.202003}
The flow (\ref{eq:1.4.202003}) is equivalent to the equation (\ref{eq:2.7.202003}) with initial $\rho(\theta, 0) = \rho_0(\theta)$
on some domain $ mS^1\times [0, t_0)$.
\end{lemma}
\begin{proof}
Let $X(\cdot, t)$ be a family of locally convex curves evolving under the flow (\ref{eq:1.4.202003}).
We immediately have the evolution equation (\ref{eq:2.7.202003}).

On the other hand, suppose we have a smooth function $\rho(\cdot, t)$ satisfying
the evolution equation (\ref{eq:2.7.202003}) with initial $\rho(\theta, 0) = \rho_0(\theta)$
which is the radius of curvature of a given locally convex curve $X_0$, where $\theta \in mS^1$.
By the continuity of $\rho$, there is a positive $t_0$
such that $\rho (\theta, t) >0$ for all $(\theta, t) \in mS^1\times [0, t_0)$.
Then one can construct a family of locally convex curves:
\begin{eqnarray}\label{eq:2.8.202003}
X(\theta, t) : = \int_0^\theta \rho(\theta, t) T(\theta) d\theta
\end{eqnarray}
where $(\theta, t) \in mS^1\times [0, t_0)$. Using the equation (\ref{eq:2.7.202003}), one can check
that the $X(\cdot, t)$, up to a parallel movement, satisfies the equation (\ref{eq:1.4.202003}). So we have done.
\end{proof}

\begin{lemma}\label{lem:2.3.202003}
The flow (\ref{eq:1.4.202003}) has a unique smooth solution in some time interval $[0, \omega)$.
\end{lemma}
\begin{proof}
The equation (\ref{eq:2.7.202003}) is half linear with a non-local term $f(t)$.
The equation is uniformly parabolic, so this Cauchy problem has a unique smooth solution in a short time interval.

Since the initial value $\rho_0$ is smooth and the higher order derivative evolves according to a linear equation
\begin{eqnarray}\label{eq:2.9.202003}
&& \frac{\partial^{k+1} \rho}{\partial \theta^k \partial t}
= \frac{\partial^{k+2} \rho}{\partial \theta^{k+2}} - \frac{\partial^{k+2} \widetilde{\rho}}{\partial \theta^{k+2}}
-\frac{\partial^k \rho}{\partial \theta^k} + \frac{\partial^k \widetilde{\rho}}{\partial \theta^k},
\end{eqnarray}
the higher order regularity of $\rho(\theta, t)$ can be obtained by the classical theory of linear parabolic equations.
Using the equation (\ref{eq:2.8.202003}), we know that the evolving curve $X(\cdot, t)$ is smooth.
\end{proof}

Under the flow (\ref{eq:1.4.202003}), the length of the evolving curve satisfies that
\begin{eqnarray*}
\frac{d L}{d t} = -L(t) + \widetilde{L} -2\pi f(t),
\end{eqnarray*}
where $\widetilde{L}$ is the length of $\widetilde{X}$ and the nonlocal term $f(t)$ is given by (\ref{eq:1.6.202003}).
In fact, the nonlocal term $f(t)$ is rather complicated. So the length $L(t)$ of the evolving curve has no
explicit solution.

\begin{remark}\label{remk:2.4.202003}
In the previous studies \cite{Gao-Zhang-2017, Lin-Tsai-2009}, the Fourier series expansion is applied to study
curvature flows. Under the flow (\ref{eq:1.4.202003}), both the evolution equation (\ref{eq:2.6.202003})
and (\ref{eq:2.7.202003}) are half linear. The nonlinear term $f(t)$ makes the method of solving linear equations with constant coefficients
by the Fourier series expansion do not work here.
\end{remark}

\begin{remark}\label{remk:2.5.202003}
Although the initial value $\rho(\theta, 0) = \rho_0(\theta)$ of the equation (\ref{eq:2.7.202003}) is positive on $mS^1$,
there is a lack of the maximum principle for the equation (\ref{eq:2.7.202003}).
Until now we do not know whether $\rho (\cdot, t)$ is always positive or not.
Once it is proved that $\rho (\cdot, t) >0$ holds for all $t\in (0, +\infty)$,
one can further show that the flow (\ref{eq:1.4.202003})
exits globally. We leave this part to the next subsection.
\end{remark}

\subsection{Global Existence}
In this subsection, it is proved that the flow (\ref{eq:1.4.202003}) exists on time interval $[0, +\infty)$.
We shall show that the curvature of the evolving curve has both uniformly lower and upper bounds and the
evolving curve $X(\cdot, t)$ is smooth for each $t \in (0, +\infty)$.

\begin{lemma}\label{lem:2.6.202003}
If the flow (\ref{eq:1.4.202003}) preserves the local convexity of the evolving curve, then the elastic energy is fixed as time goes.
\end{lemma}
\begin{proof}
Under the flow (\ref{eq:1.4.202003}), the curvature $\kappa(\theta, t)$ of the curve $X(\cdot, t)$ satisfies
the equation (\ref{eq:2.4.202003}). So the elastic energy of the evolving curve satisfies that
\begin{eqnarray*}
\frac{dE}{dt} &=& \frac{d}{dt} \int_{X(\cdot, t)} \kappa^2(s, t) ds
= \frac{d}{dt} \int_{mS^1}\kappa(\theta, t) d\theta
\\
&=& \int_{mS^1} \kappa^2 \left(-\frac{\partial^2 \rho}{\partial \theta^2} + \frac{\partial^2 \widetilde{\rho}}{\partial \theta^2}
+ \rho - \widetilde{\rho} +f(t)\right) d\theta¡£
\end{eqnarray*}
By the choice of the nonlocal term $f(t)$, we have $\frac{dE}{dt} \equiv 0$.
\end{proof}

To make the statement brief in the following proofs, we introduce two functions for $t \geq 0$.
Now define
$$\rho_{\max} (t) := \max\{\rho(\theta, t)| \theta \in mS^1\}, ~~
\rho_{\min} (t) := \min\{\rho(\theta, t)| \theta \in mS^1\}.$$

\begin{corollary}\label{cor:2.7.202003}
If the flow (\ref{eq:1.4.202003}) preserves the local convexity of the evolving curve, then
\begin{eqnarray}\label{eq:2.10.202003}
\rho_{\min} (t) \leq \frac{2m\pi}{E} \leq \rho_{\max} (t).
\end{eqnarray}
\end{corollary}
\begin{proof}
The elastic energy is preserved under the flow (\ref{eq:1.4.202003}), so
\begin{eqnarray*}
E= \int_{mS^1} \kappa(\theta, t) d\theta = \int_{mS^1} \frac{1}{\rho(\theta, t)} d\theta
\leq \frac{2m\pi}{\rho_{\min} (t)},
\end{eqnarray*}
i.e., $\rho_{\min} (t) \leq \frac{2m\pi}{E}$. Similarly, we have $\frac{2m\pi}{E} \leq \rho_{\max} (t)$.
\end{proof}

The following Harnack estimate is a key step towards the proof of the global existence of the flow (\ref{eq:1.4.202003}).
By this estimate, one may control the curvature uniformly.

\begin{lemma}\label{lem:2.8.202003} (Harnack estimate)
If the flow (\ref{eq:1.4.202003}) preserves the local convexity of the evolving curve, then there exists a constant $H$
independent of time such that
\begin{eqnarray}\label{eq:2.11.202003}
\rho_{\max} (t) \leq \rho_{\min} (t)\cdot H.
\end{eqnarray}
\end{lemma}
\begin{proof}
Set $u(\theta, t) = \rho(\theta, t) - \widetilde{\rho}(\theta)$. By the evolution equation of $\rho$ (or see the
equation (\ref{eq:2.9.202003})),
$\frac{\partial^i u}{\partial \theta^i}$ satisfies
\begin{eqnarray*}
\frac{\partial}{\partial t} \frac{\partial^i u}{\partial \theta^i}
= \frac{\partial^{i+2} u}{\partial \theta^{i+2}} - \frac{\partial^i u}{\partial \theta^i}.
\end{eqnarray*}
So the function $\left(\frac{\partial^i u}{\partial \theta^i}\right)^2$ evolves according to
\begin{eqnarray*}
\frac{\partial}{\partial t} \left(\frac{\partial^i u}{\partial \theta^i}\right)^2
&=& \frac{\partial^2}{\partial \theta^2}\left(\frac{\partial^i u}{\partial \theta^i}\right)^2
-2\left(\frac{\partial^{i+1} u}{\partial \theta^{i+1}}\right)^2-2\left(\frac{\partial^i u}{\partial \theta^i}\right)^2
\\
&\leq & = \frac{\partial^2}{\partial \theta^2}\left(\frac{\partial^i u}{\partial \theta^i}\right)^2
-2\left(\frac{\partial^i u}{\partial \theta^i}\right)^2.
\end{eqnarray*}
Applying the maximum principle, one obtains that
\begin{eqnarray} \label{eq:2.12.202003}
\left(\frac{\partial^i u}{\partial \theta^i}\right)^2 \leq C_i e^{-2t},
\end{eqnarray}
where $$C_i = \max\left\{\left(\frac{\partial^i u}{\partial \theta^i} (\theta, 0)\right)^2 \bigg| \theta\in mS^1\right\}$$
is a constant depending on the initial curve $X_0$.
Since $$\left|\frac{\partial^i \rho}{\partial \theta^i}\right|
\leq \left|\frac{\partial^i \widetilde{\rho}}{\partial \theta^i}\right| +\left|\frac{\partial^i u}{\partial \theta^i}\right|
\leq \left|\frac{\partial^i \widetilde{\rho}}{\partial \theta^i}\right| +\sqrt{C_i},$$
there is a constant, denoted by $M_i$, independent of time such that
\begin{eqnarray} \label{eq:2.13.202003}
\left|\frac{\partial^i \rho}{\partial \theta^i}\right| \leq M_i,
\end{eqnarray}
where $i= 1, 2, 3, \cdots$.

Fixed the time $t$. Suppose the function $\rho(\theta, t)$ attains its minimum and maximum at points $\theta_1, \theta_2$,
respectively. Compute
\begin{eqnarray*}
\ln \rho_{\max} (t) - \ln \rho_{\min} (t) = \int_{\theta_1}^{\theta_2} \frac{1}{\rho} \frac{\partial \rho}{\partial \theta} d \theta
 \leq  \int_{mS^1} \frac{1}{\rho} \left|\frac{\partial \rho}{\partial \theta}\right| d \theta
  \leq  M_1 E,
\end{eqnarray*}
where the constant $M_1$ is given by (\ref{eq:2.13.202003}).
Therefore, choosing $H= e^{M_1 E}$ can give us the estimate (\ref{eq:2.11.202003}).
\end{proof}

\begin{lemma}\label{lem:2.9.202003}
The flow (\ref{eq:1.4.202003}) preserves the local convexity of the evolving curve.
\end{lemma}
\begin{proof}
By the continuity of the evolving curve, there exists $t_0 > 0$ such that $X(\cdot, t)$ is locally convex on the time interval
$[0, t_0)$. In this same time interval, it follows from Lemma \ref{lem:2.8.202003} and Corollary \ref{cor:2.7.202003},
\begin{eqnarray} \label{eq:2.14.202003}
\rho_{\max} (t) \leq \rho_{\min} (t) H \leq \frac{2m\pi}{E} H.
\end{eqnarray}
So the curvature of the evolving curve satisfies that
\begin{eqnarray*}
\kappa(\theta, t) \geq \kappa_{\min}(t) = \frac{1}{\rho_{\max} (t)} \geq \frac{E}{2m\pi H} >0.
\end{eqnarray*}
As the flow exists, the evolving curve is always locally convex.
\end{proof}

\begin{theorem}\label{thm:2.10.202003}
The flow (\ref{eq:1.4.202003}) exists on the time interval $[0, +\infty)$.
\end{theorem}
\begin{proof}
By (\ref{eq:2.13.202003}), all higher derivatives $\frac{\partial^i \rho}{\partial \theta^i}$ is uniformly bounded.
It suffices to show that $\rho(\theta, t)>0$ for all $(\theta, t)\in mS^1 \times [0, +\infty)$.
The Harnack estimate  (\ref{eq:2.14.202003}) and the second inequality of (\ref{eq:2.3.202003}) imply that
\begin{eqnarray} \label{eq:2.15.202003}
\rho_{\min}(t) \geq \rho_{\max} (t) \frac{1}{H} \geq \frac{2m\pi}{EH} >0.
\end{eqnarray}
Hence, the curvature $\kappa = \frac{1}{\rho}$ never blows up as time goes.
Until now we have shown that the function $\rho(\theta, t)$ is uniformly bounded and positive on the domain $mS^1 \times [0, +\infty)$.
So the velocity of the flow  (\ref{eq:1.4.202003}) $\frac{\partial X}{\partial t}$ is smooth for all $t>0$.
Integrating the equation (\ref{eq:1.4.202003})
\begin{eqnarray*}
X(\theta, t) = X_0(\theta) + \int_0^t \frac{\partial X}{\partial t} (\theta, \widetilde{t}) d\widetilde{t},
\end{eqnarray*}
we know the evolving curve $X(\cdot, t)$ is smooth for every $t\in [0, +\infty)$.
The flow (\ref{eq:1.4.202003}) exists globally.
\end{proof}

There are two key steps in the proof of Theorem \ref{thm:2.10.202003}.
The first one is that the derivative $\frac{\partial \rho}{\partial \theta}$
has a uniform bound under the flow (see the equation (\ref{eq:2.13.202003})). This fact holds because we have the term
$2p(\theta, t) - 2\widetilde{p}(\theta, t)$ in the flow equation (\ref{eq:1.4.202003}).
The second step is that the flow preserves the elastic energy.
The construction of the nonlocal term $f(t)$ in the flow equation guarantees this property.

\section{Convergence}
We first prove the convergence of the radius of curvature. Then we show the convergence of the support function, which implies
the convergence of the evolving curve $X(\cdot, t)$ as the time $t \rightarrow +\infty$.

\begin{lemma}\label{lem:3.1.202003}
Under the flow (\ref{eq:1.4.202003}), the radius of the curvature $\rho (\cdot, t)$ of the evolving curve converges
as $t \rightarrow +\infty$. If $X_0$ and $\widetilde{X}$ have same elastic energy then
\begin{eqnarray} \label{eq:3.1.202003}
\lim_{t \rightarrow +\infty} \rho(\theta, t) = \widetilde{\rho}(\theta).
\end{eqnarray}
\end{lemma}
\begin{proof}
By (\ref{eq:2.11.202003}), (\ref{eq:2.13.202003}) and (\ref{eq:2.14.202003}),
both $\rho$ and $|\frac{\partial \rho}{\partial \theta}|$ are uniformly bounded by constants.
There exits a convergent subsequence $\rho(\theta, t_i)$ as $t_i \rightarrow +\infty$. Let $\rho_{\infty}(\theta)$ be the limit
of $\rho(\theta, t_i)$. By (\ref{eq:2.13.202003}), $\rho_{\infty}(\theta)$ is smooth and for any positive integer $k$,
\begin{eqnarray} \label{eq:3.2.202003}
\lim_{t_i \rightarrow +\infty} \frac{\partial^k \rho}{\partial \theta^k}(\theta, t_i)
=  \frac{\partial^k \rho_\infty}{\partial \theta^k}(\theta).
\end{eqnarray}
It follows from (\ref{eq:2.12.202003}) that, for every fixed $\theta \in mS^1$,
\begin{eqnarray} \label{eq:3.3.202003}
\lim\limits_{t \rightarrow +\infty} \frac{\partial \rho}{\partial \theta}(\theta, t)
=  \frac{\partial \widetilde{\rho}}{\partial \theta}(\theta)
\end{eqnarray}
So, combining (\ref{eq:3.2.202003}) and (\ref{eq:3.3.202003}), we have
\begin{eqnarray*}
\frac{\partial \rho_\infty}{\partial \theta}(\theta) = \lim_{t_i \rightarrow +\infty} \frac{\partial \rho}{\partial \theta}(\theta, t_i)
=  \frac{\partial \widetilde{\rho}}{\partial \theta}(\theta).
\end{eqnarray*}
By (\ref{eq:2.7.202003}) and (\ref{eq:2.10.202003}),  $\rho_\infty(\theta)$ is uniformly bounded. There is a constant $c_0$ such that
\begin{eqnarray} \label{eq:3.4.202003}
\rho_\infty(\theta) = \widetilde{\rho}(\theta) + c_0,
\end{eqnarray}
where $\theta \in mS^1$.
Recall that the flow (\ref{eq:1.4.202003}) preserves the elastic energy $E$. The constant $c_0$ is uniquely determined by
\begin{eqnarray} \label{eq:3.5.202003}
E = \int_{mS^1} \frac{1}{\widetilde{\rho}(\theta) + c_0} d\theta.
\end{eqnarray}
Since every convergent subsequence of $\rho(\cdot, t)$ tends to the fixed limit $\widetilde{\rho}(\cdot) + c_0$,
$\rho(\cdot, t)$ itself converges to the same limit:
\begin{eqnarray} \label{eq:3.6.202003}
\lim_{t \rightarrow +\infty} \rho(\theta, t) = \widetilde{\rho}(\theta) + c_0.
\end{eqnarray}

Suppose $X_0$ and $\widetilde{X}$ have same elastic energy. Noticing that the flow (\ref{eq:1.4.202003}) preserves the elastic energy,
one gets from (\ref{eq:3.4.202003}) and (\ref{eq:3.5.202003}) that
\begin{eqnarray} \label{eq:3.7.202003}
\int_{mS^1} \frac{1}{\widetilde{\rho}(\theta) + c_0} d\theta
=\int_{mS^1} \frac{1}{\rho_{\infty}(\theta)} d\theta = E
=\int_{mS^1} \frac{1}{\rho_0(\theta)} d\theta
= \int_{mS^1} \frac{1}{\widetilde{\rho}(\theta)} d\theta.
\end{eqnarray}
Comparing the both sides, one obtains that the constant $c_0 =0$.
By (\ref{eq:3.7.202003}), $\rho(\theta, t)$ converges to $\widetilde{\rho}(\theta)$
as $t \rightarrow +\infty$. It follows from  (\ref{eq:2.13.202003}) and (\ref{eq:2.12.202003}) that this convergence
is in the sense of $C^\infty$ metric, i.e., we have the limit
\begin{eqnarray*}
\lim_{t \rightarrow +\infty} \frac{\partial^i \rho}{\partial \theta^i}(\theta, t)
= \frac{\partial^i \widetilde{\rho}}{\partial \theta^i}(\theta).
\end{eqnarray*}
for each integer $i = 1, 2, \cdots$.
\end{proof}

As a corollary of the convergence of $\rho(\theta, t)$, we know that the nonlocal term $f(t)$ in the flow
(\ref{eq:1.4.202003}) converges:
\begin{eqnarray*}
\lim_{t \rightarrow +\infty} f(t) = 0.
\end{eqnarray*}

\begin{theorem}\label{thm:3.2.202003}
The evolving curve under the flow (\ref{eq:1.4.202003}) converges to the target curve $\widetilde{X}$,
if $X_0$ and $\widetilde{X}$ have same elastic energy.
\end{theorem}
\begin{proof}
By (\ref{eq:1.2.202003}) and (\ref{eq:1.3.202003}), the evolving curve $X(\cdot, t)$ is uniquely determined by its
support function
\begin{eqnarray*}
X(\theta, t) =\frac{\partial p}{\partial \theta}(\theta, t) T(\theta, t) - p(\theta, t) N(\theta, t).
\end{eqnarray*}
So it suffices to prove the convergence of $p(\cdot, t)$.

Under the flow (\ref{eq:1.4.202003}), the support function of the evolving curve satisfies the
equation (\ref{eq:2.6.202003}). Mimicing the proof of Lemma \ref{lem:3.1.202003},
one can show that
\begin{eqnarray} \label{eq:3.8.202003}
\lim_{t \rightarrow +\infty} \frac{\partial^i p}{\partial \theta^i}(\theta, t)
=  \frac{\partial^i \widetilde{p}}{\partial \theta^i}(\theta),
\end{eqnarray}
where $i= 1, 2, 3, \cdots$.
Therefore the equation (\ref{eq:3.7.202003}) together with the relation between $p$ and $\rho$ imply that
\begin{eqnarray} \label{eq:3.9.202003}
\lim_{t \rightarrow +\infty} p(\theta, t) = \widetilde{p}(\theta) + c_0,
\end{eqnarray}
where constant $c_0$ is uniquely determined by (\ref{eq:3.5.202003}).
Combining (\ref{eq:3.8.202003}) with (\ref{eq:3.9.202003}), we know that $p(\cdot, t)$
converges to $\widetilde{p}(\theta, t)$ in the sense of $C^\infty$ metric.

If $X_0$ and $\widetilde{X}$ have same elastic energy then it follows from (\ref{eq:3.7.202003}) that $c_0 =0$.
The convergence of the support function implies that the evolving curve $X(\cdot, t)$ converges to
the curve $\widetilde{X}$:
\begin{eqnarray*}
\lim_{t \rightarrow +\infty} X(\theta, t) = \widetilde{X}(\theta).
\end{eqnarray*}
The higher order convergence of the evolving curve follows from the
$C^\infty$ convergence of the support function directly.
\end{proof}

Combining Theorem \ref{thm:2.10.202003} with Theorem \ref{thm:3.2.202003}, one immediately proves
Theorem \ref{thm:1.2.202003}.

\begin{corollary}\label{lem:3.3.202003}
Let $X_0$ be an initial locally convex curve with winding number $m$ and elastic energy $E$.
Let $Y_0$ be a target convex curve with elastic energy $\frac{E}{m}$. The flow (\ref{eq:1.4.202003}) can deform
$X_0$ into the $m$-fold convex curve $mY_0$ as $t \rightarrow +\infty$.
\end{corollary}

\begin{remark}\label{rem:3.4.202003}
Let $X_0$ and $Y_0$ be any two locally convex curves with same wingding number. There is a proper rescaling of $Y_0$,
denoted by $\widetilde{Y}$, such that $X_0$ and $\widetilde{Y}$ have same elastic energy.
By Theorem \ref{thm:1.2.202003}, $X_0$ can be deformed into $\widetilde{Y}$
under the flow (\ref{eq:1.4.202003}). Therefore, up to a rescaling,
any two locally convex curves with same wingding number can be evolved into
each other via the curvature flow (\ref{eq:1.4.202003}).
\end{remark}

\section{An Example}
In this section, an example of the flow (\ref{eq:1.4.202003}) is investigated.
Let $\widetilde{X}$ be a locally convex curve with support function
$$\widetilde{p} = 10 + 10\sin\left(\frac{2}{3}\theta\right)+10\cos\left(\frac{2}{3}\theta\right),$$
where $\theta \in [0, 6\pi]$.
With the help of MATLAB, one can calculate that the minimum value of its radius of the curvature is 1.7725.
We choose a locally convex curve $X_0$ with support function
$$p_0 = 11.6720 + 9\sin\left(\frac{4}{3}\theta\right)+9\cos\left(\frac{4}{3}\theta\right),$$
where $\theta \in [0, 6\pi]$.
The minimum value of its radius of the curvature is  2.1433.
The two curves have same winding number 3 and same elastic energy 3.0463. See the figures (up to proper
rotations) of the two curves in Figure \ref{fig:1}.
\begin{figure}[tbh]
\centering
\subfloat[The Initial Curve $X_0$]{\scalebox{0.4}{\includegraphics{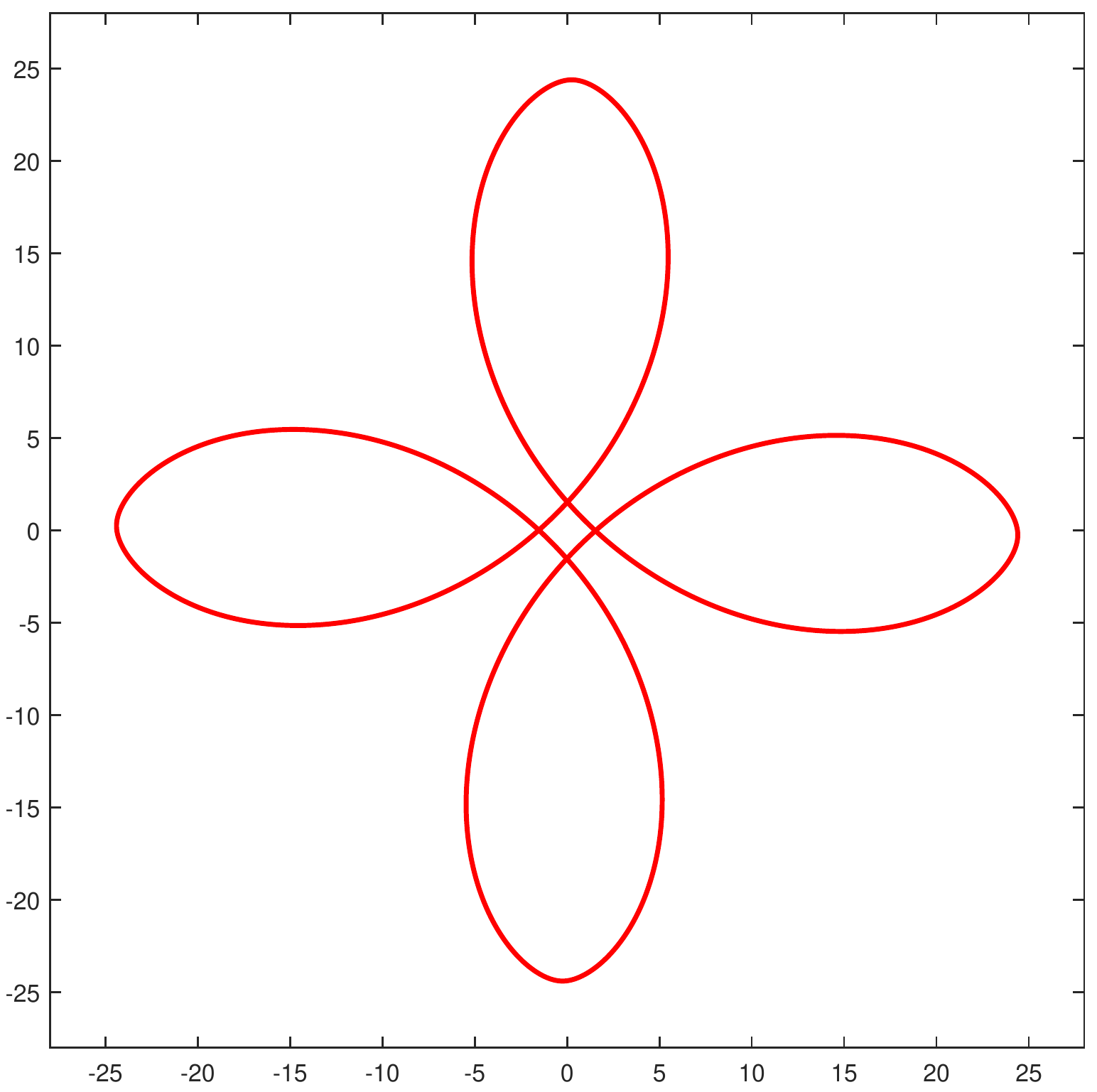}}}~~~
\subfloat[The Target Curve $\widetilde{X}$]{\scalebox{0.4}{\includegraphics{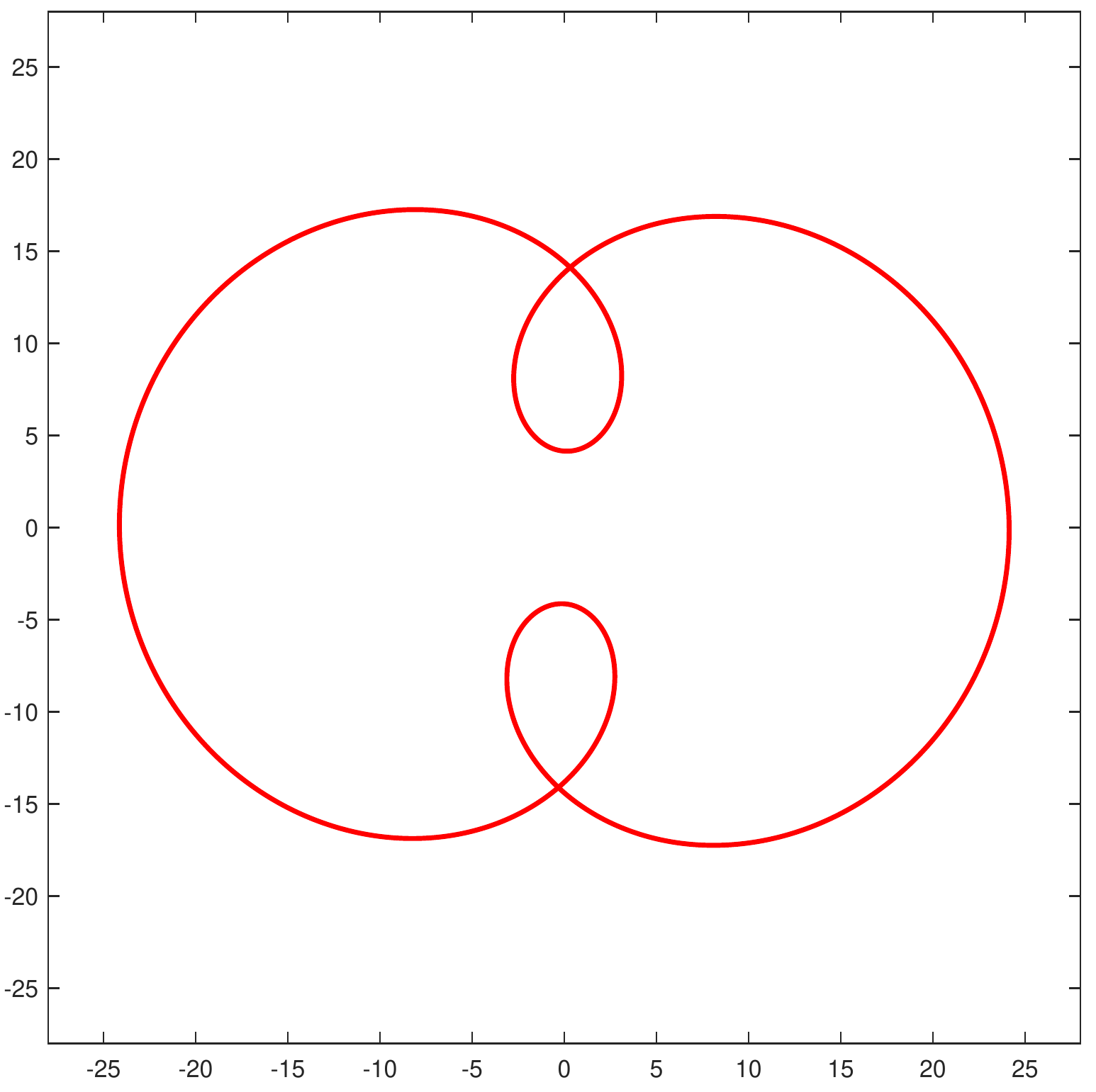}}}
\caption{} \label{fig:1}
\end{figure}

\begin{example}\label{rem:3.3.202003}
Let $\widetilde{X}$ be the target curve and let $X_0$ evolve according to the flow (\ref{eq:1.4.202003}).
By Theorem \ref{thm:1.2.202003}, the evolving curve $X(\cdot, t)$ is smooth on the time interval $[0, +\infty)$ and
it converges to the curve $\widetilde{X}$ as $t \rightarrow +\infty$. Some evolving curves (up to proper
rotations) are presented in Figure \ref{fig:2.202002} and some relative geometric quantities are given in Table \ref{tab:1}.
\end{example}
\begin{table}[tbh]
  \centering   \caption{Some Geometric Quantities of the Evolving Curve}\label{tab:1}
  \begin{tabular*}{10.1cm}{|c|c|c|c|c|}
     \hline
     Time $t$   & $\frac{L(t)}{6\pi}$       & $\rho_{\min}(t)$   & $\rho_{\max}(t)$  & Elastic Energy\\
     \hline
     0             & 11.6720                  & 1.7725           &  21.5715          & 3.0463\\
     \hline
     0.01         & 11.44718                  & 1.7148           &  21.1187          & 3.0463\\
     \hline
     0.05         & 10.69347                  & 1.5713           &  19.5224          & 3.0463\\
     \hline
     0.1          & 9.9949                    & 1.5175           &  17.9135          & 3.0463\\
     \hline
     0.2          & 9.0992                    & 1.5865           &  15.6041          & 3.0463\\
     \hline
     0.4          & 8.3953                    & 1.8973           &  13.3267          & 3.0463\\
     \hline
     0.6          & 8.3667                    & 2.1597           &  12.9290          & 3.0463\\
     \hline
     1            & 8.8301                    & 2.3431           &  14.4828          & 3.0463\\
     \hline
     2            & 9.6740                    & 2.2271           &  17.0667          & 3.0463\\
     \hline
     4            & 9.9810                    & 2.1485           &  17.8133          & 3.0463\\
     \hline
     $+\infty$    & 10                        & 2.1433           &  17.8567          & 3.0463\\
     \hline
  \end{tabular*}
\end{table}

\begin{figure}[tbh]
\centering
\subfloat[$t=0.01$]{\scalebox{0.32}{\includegraphics{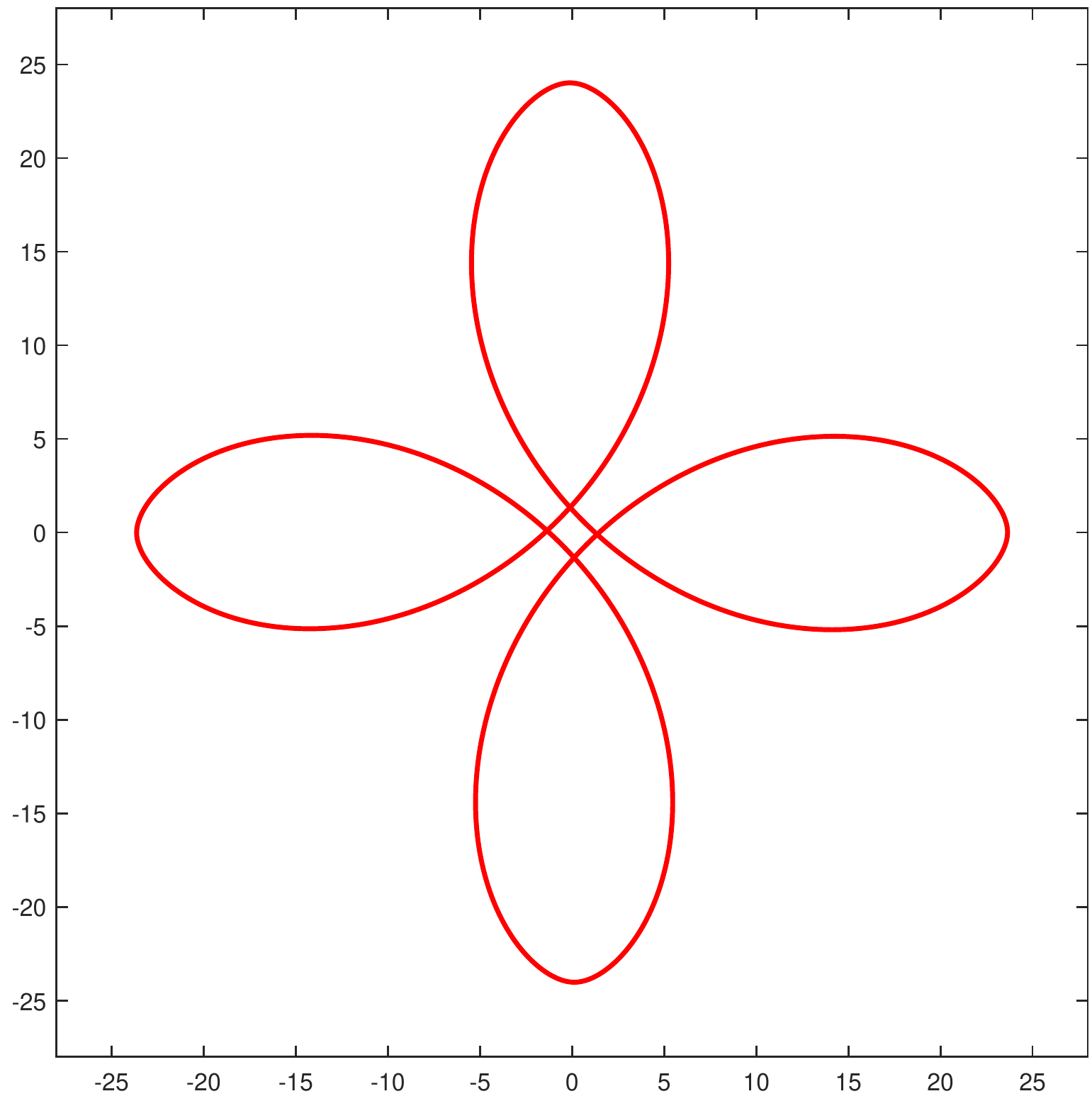}}}
\subfloat[$t=0.05$]{\scalebox{0.32}{\includegraphics{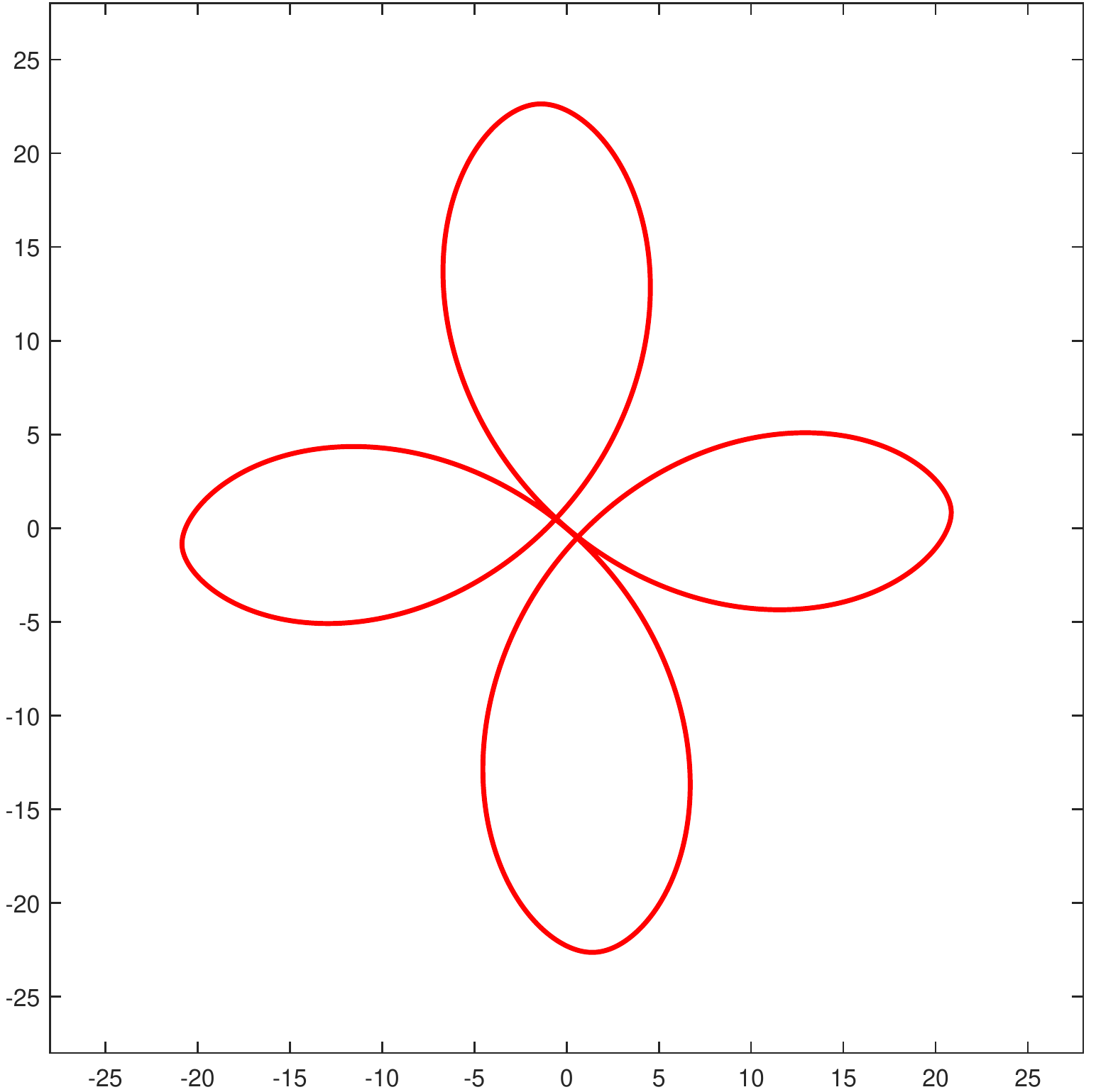}}}
\subfloat[$t=0.1$ ]{\scalebox{0.32}{\includegraphics{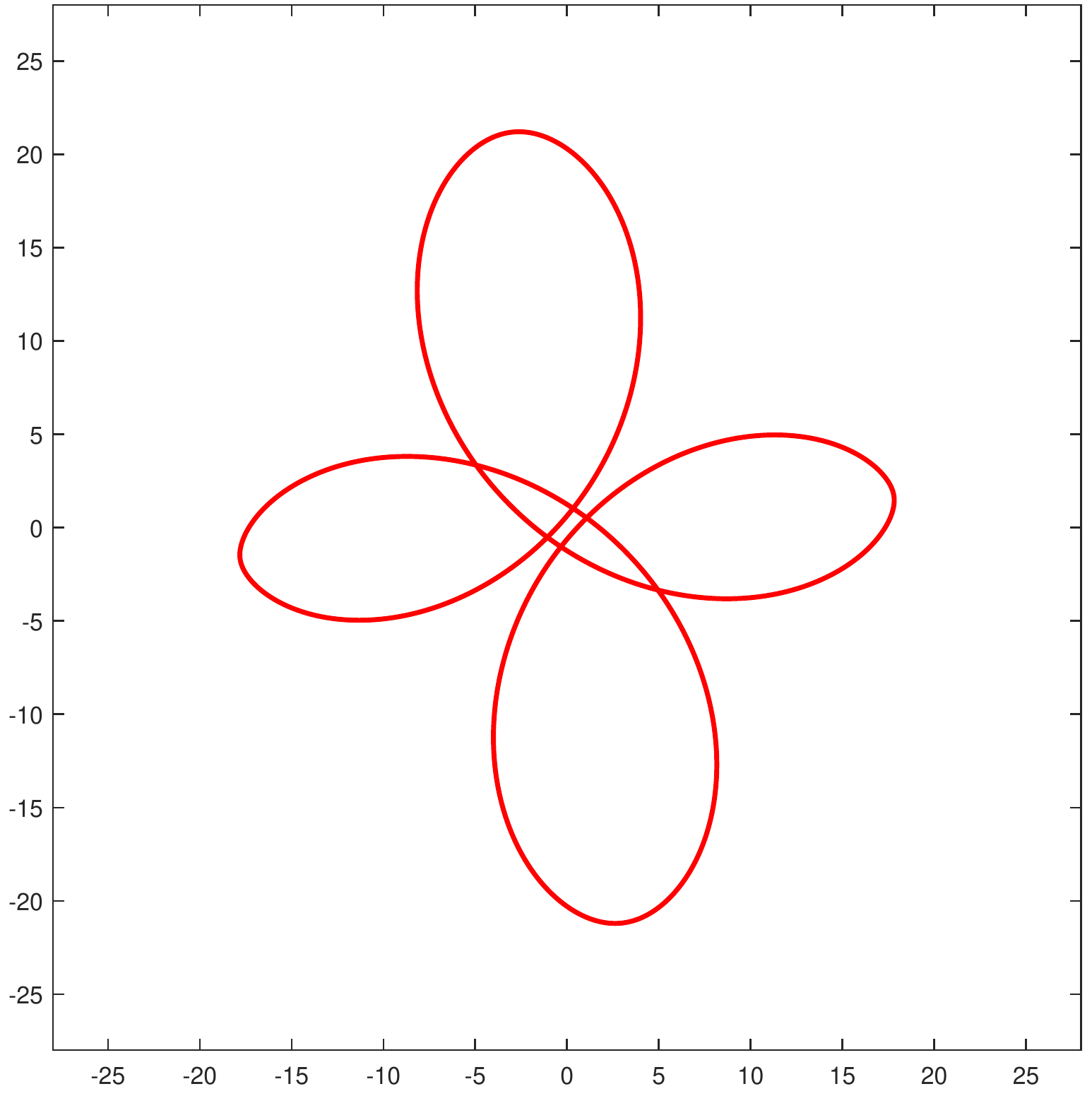}}}\
\subfloat[$t=0.2$ ]{\scalebox{0.32}{\includegraphics{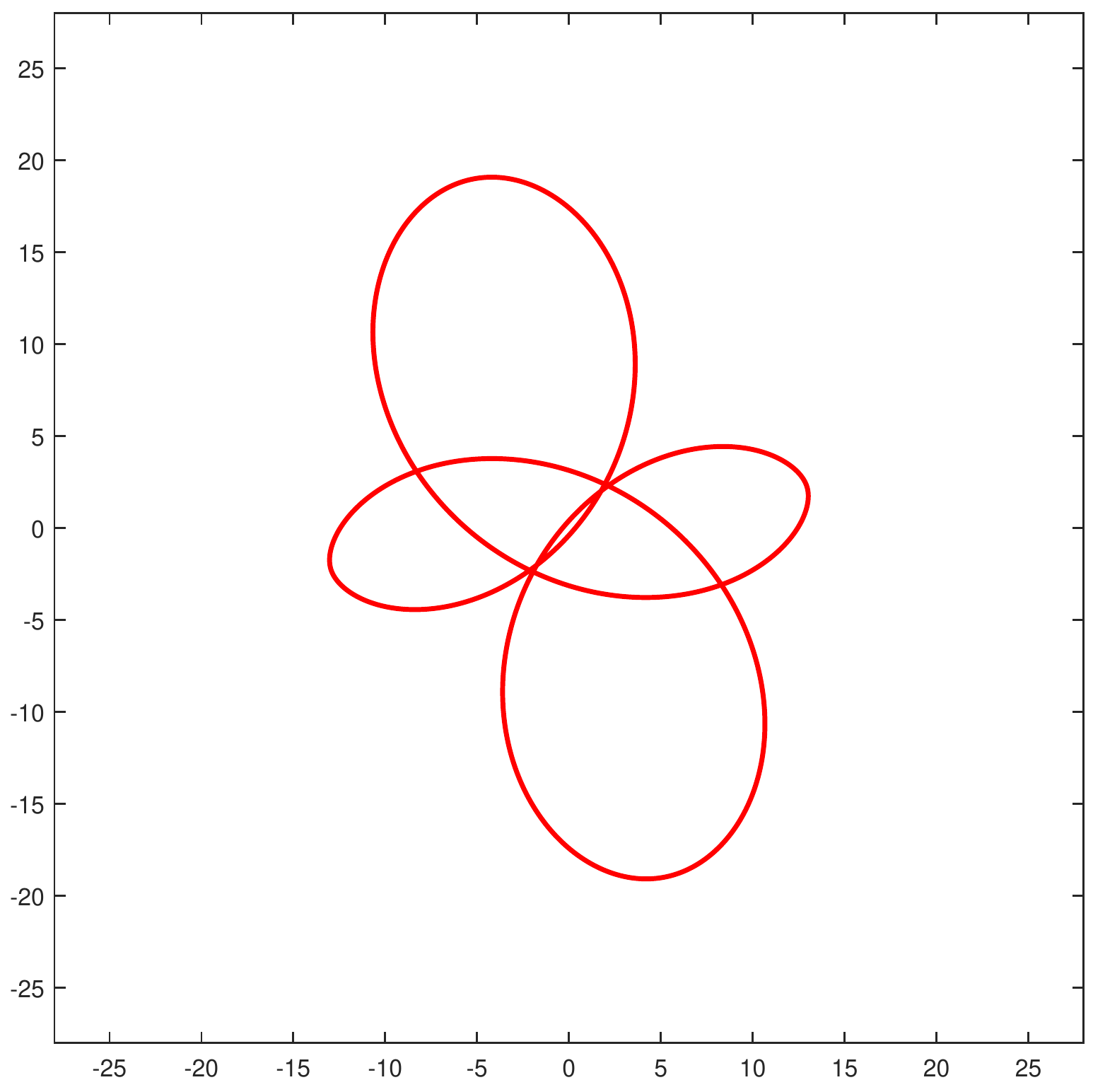}}}
\subfloat[$t=0.4$ ]{\scalebox{0.32}{\includegraphics{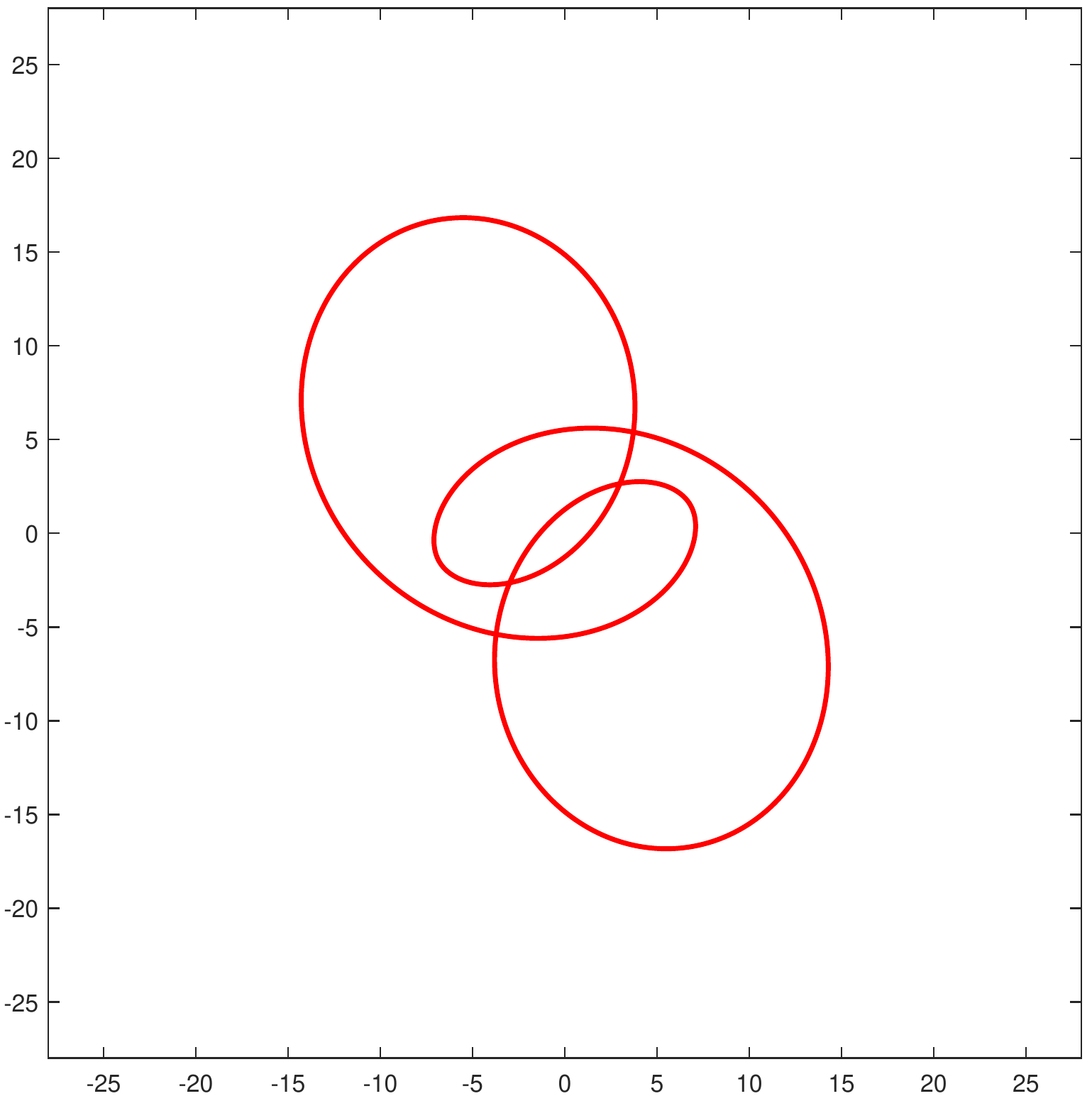}}}
\subfloat[$t=0.6$ ]{\scalebox{0.32}{\includegraphics{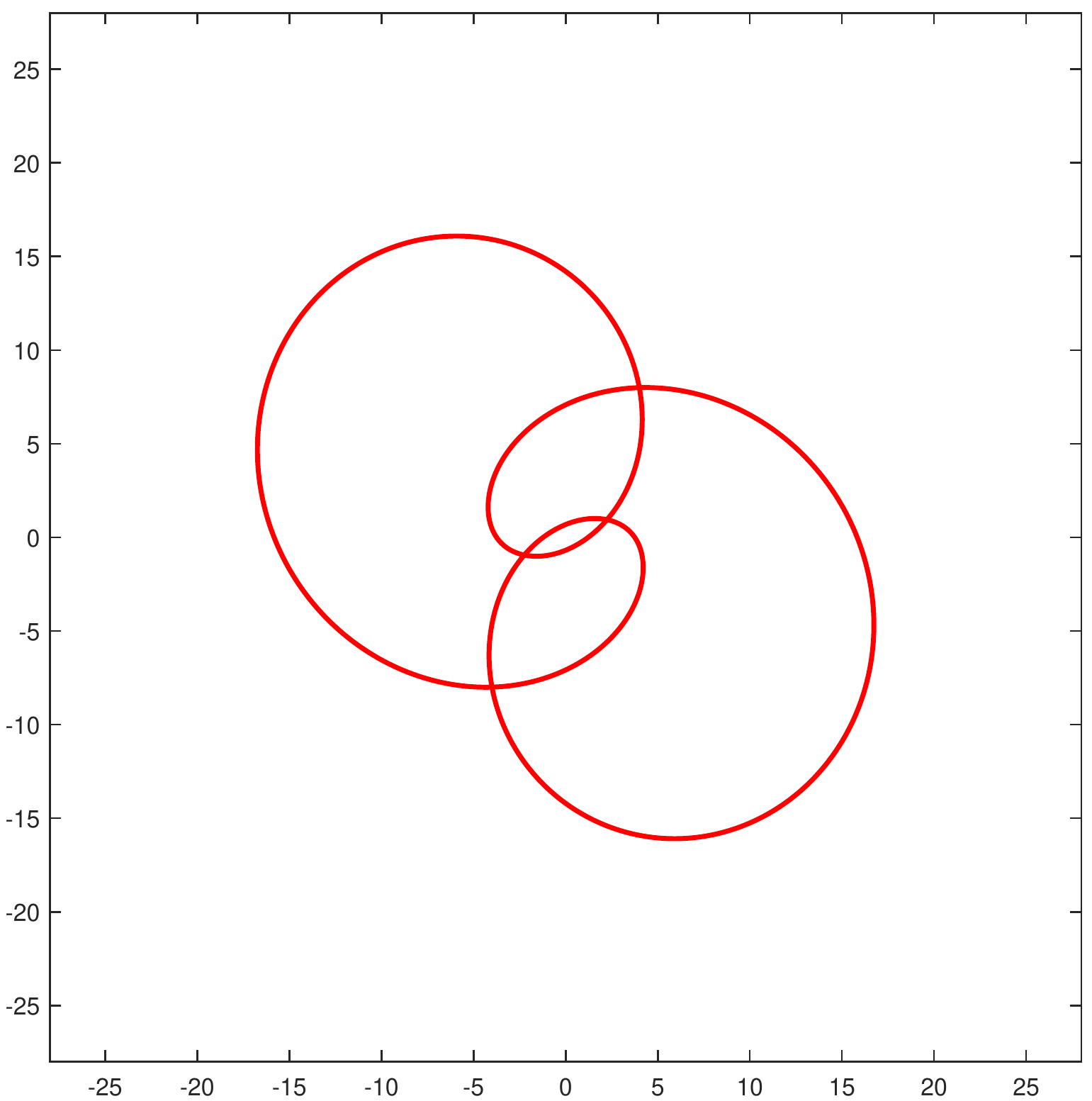}}}\
\subfloat[$t=1$ ]{\scalebox{0.32}{\includegraphics{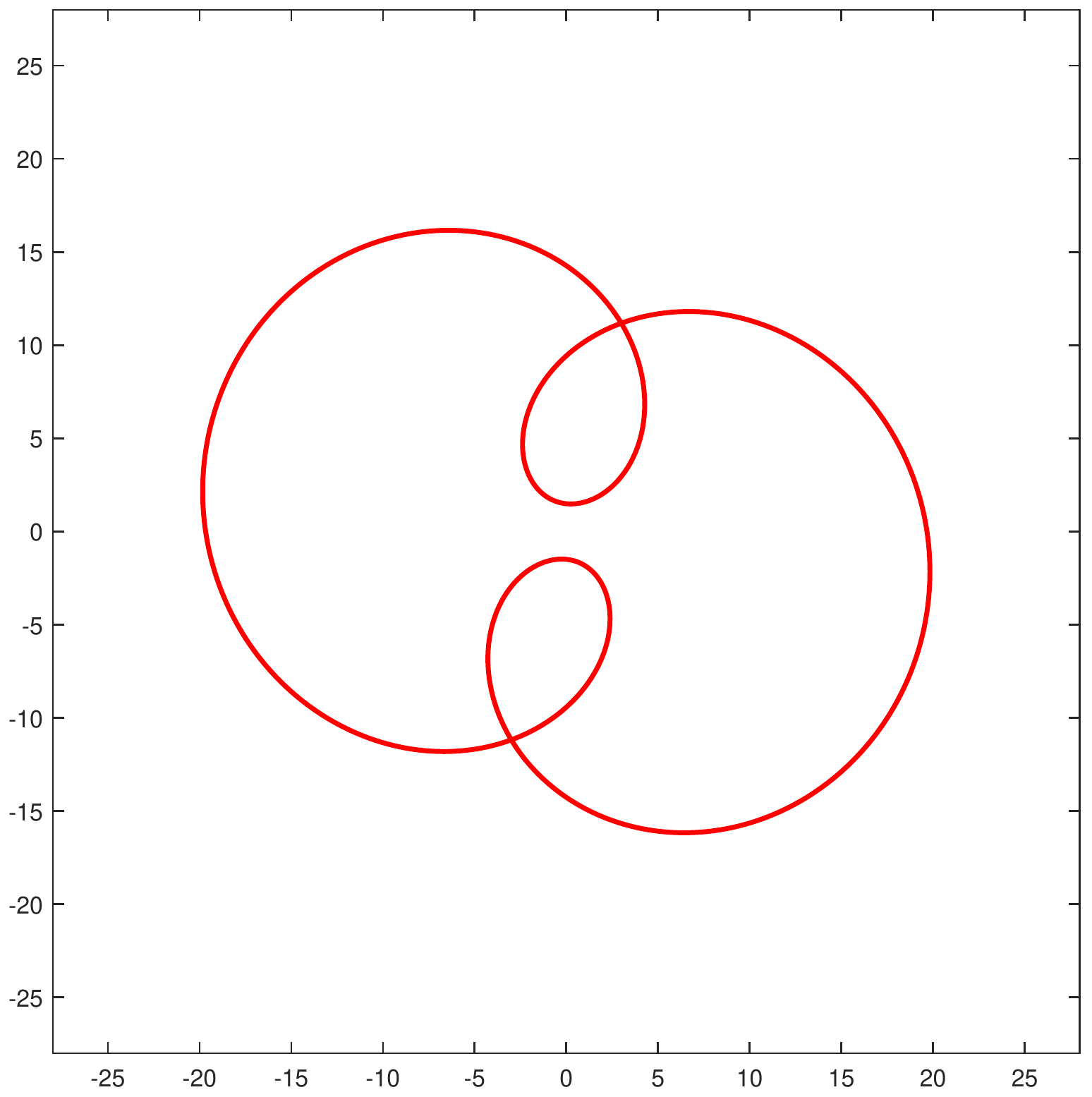}}}
\subfloat[$t=2$ ]{\scalebox{0.32}{\includegraphics{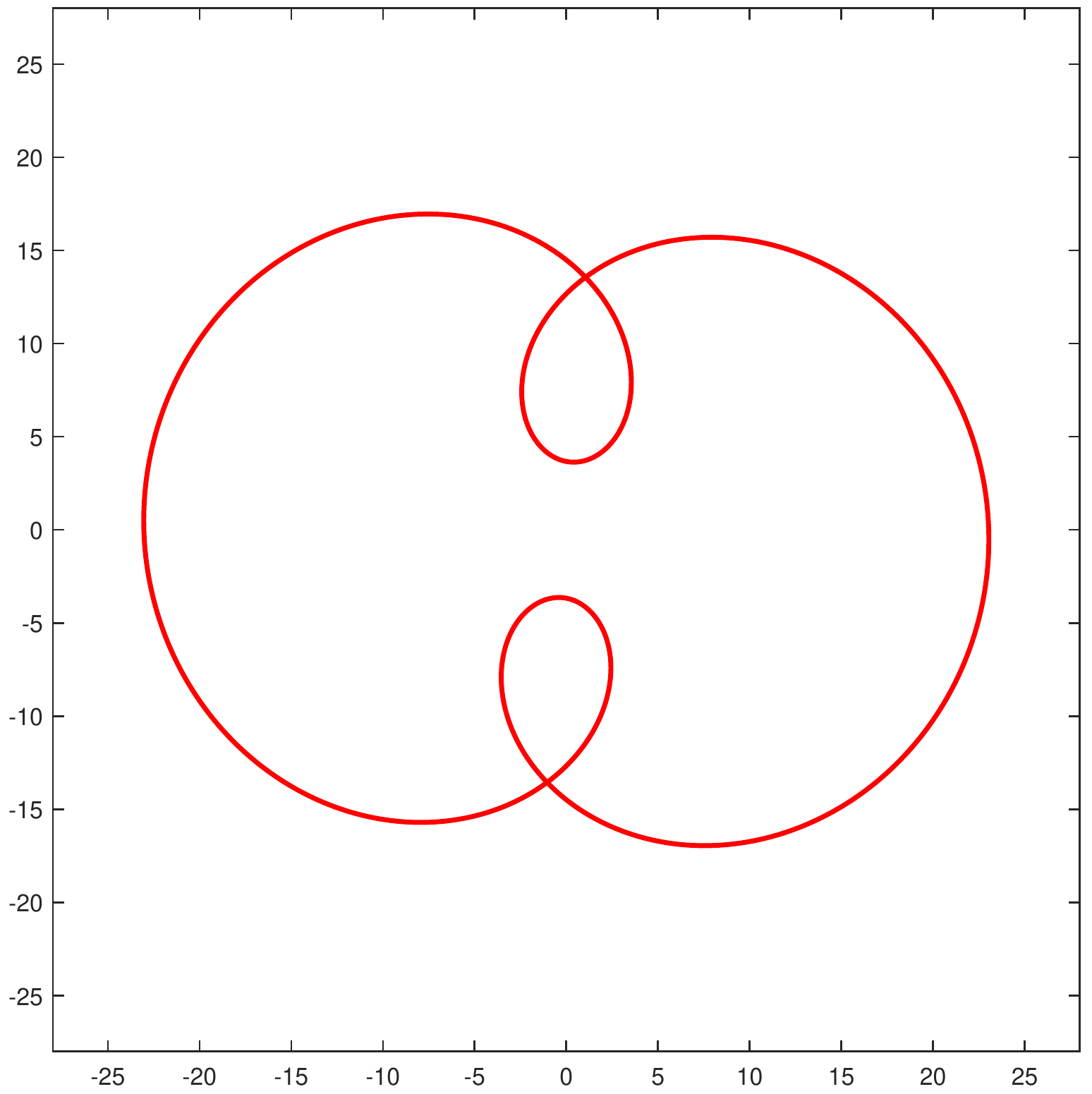}}}
\subfloat[$t=4$ ]{\scalebox{0.32}{\includegraphics{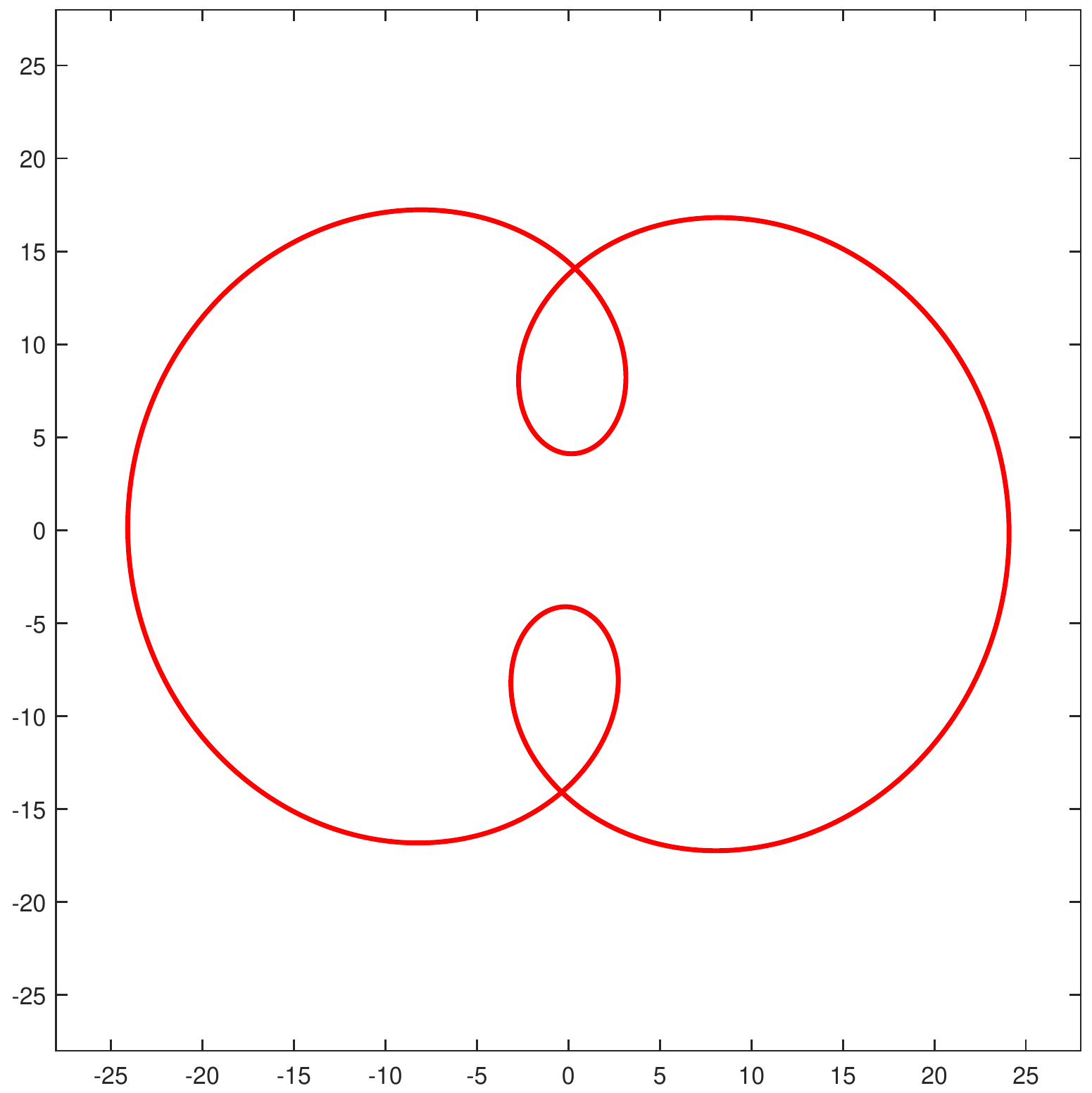}}}
\caption{Some Time Shots of the Evolution} \label{fig:2.202002}
\end{figure}

~\\
\textbf{Acknowledgments}
Laiyuan Gao is supported by National Natural Science Foundation of China (No.11801230).
This work is completed when Gao visited University of California, San Diego.
He thanks Prof. Lei Ni and Prof. Bennett Chow for their hospitality.

{\bf Laiyuan Gao}

School of Mathematics and Statistics, Jiangsu Normal University.

No.101, Shanghai Road, Xuzhou City, Jiangsu Province, China.

Email: lygao@jsnu.edu.cn\\\\


\begin{thebibliography}{99}

\bibitem{Abresch-Langer-1986} U. Abresch, J. Langer,
The normalized curve shortening flow and homothetic solutions.
J. Differential Geom. No. 2, Vol. 34, 23 (1986), 175-196.

\bibitem{Altschuler-1991} S. J. Altschuler,
Singularities of the curve shrinking flow for space curves.
J. Differential Geom. Vol. 34 (1991), 491-514.

\bibitem{Andrews-2002} B. Andrews,
Classification of limiting shapes for isotropic curve flows.
Journal of the American Mathematical Society No. 2, Vol. 16 (2002), 443-459.

\bibitem{Angenent-1991} S. Angenent,
On the formation of singularities in the curve shortening flow.
J. Differential Geom. No. 3, Vol. 33 (1991), 601-633.

\bibitem{Angenent-2005} S. Angenent,
Curve shortening and the topology of closed geodesics on surfaces.
Annals of Mathematics 162 (2005), 1187-1241.

\bibitem{Chen-Wang-Yang-2017} W.-Y. Chen, X.-L. Wang, M. Yang,
Evolution of highly symmetric curves under the shrinking curvature flow.
Math. Methods Appl. Sci. No. 10, Vol. 40 (2017), 3775-3783.

\bibitem{Chow-Liou-Tsai} B. Chow, L.-P. Liou, D.-H. Tsai,
Expansion of embedded curves with turning angle greater than $-\pi$.
Invent. Math. Vol. 123 (1996), 415-429.

\bibitem{Chou-Zhu-1999-1} K.-S. Chou \& X.-P. Zhu,
Anisotropic flows for convex plane curves.
Duke Math. J. No. 3, Vol. 97 (1999), 579-619.

\bibitem{Chou-Zhu-1999-2} K.-S. Chou \& X.-P. Zhu,
A convexity theorem for a class of anisotropic flows of plane curves.
Indiana Univ. Math. J. 48(1) (1999), 139-154.


\bibitem{Epstein-Gage-1987} C. L. Epstein \& M. Gage,
The curve shortening flow.
Wave Motion: Theory, Modeling and Computation, A Chorin and A Majda,
Editors, Springer-Verlag, New York, 1987.


\bibitem{Gage-1983} M. E. Gage,
An isoperimetric inequality with applications to curve shortening.
Duke Math. J. No. 4, Vol. 50 (1983), 1225-1229.

\bibitem{Gage-1984} M. E. Gage,
Curve shortening makes convex curves circular.
Invent. Math. No. 2, Vol. 76, (1984), 357-364.

\bibitem{Gage-1986} M. E. Gage,
On an area-preserving evolution equation for plane curves.
in: D.M. DeTurck (Ed.), Nonlinear Problems in Geometry,
in: Contemp. Math. Vol. 51 (1986), 51-62.

\bibitem{Gage-1993} M. E. Gage,
Evolving plane curves by curvature in relative geometries.
Duke Math. J. No. 2, Vol. 72 (1993), 441-466.

\bibitem{Gage-Hamilton-1986} M. E. Gage \& R. S. Hamilton,
The heat equation shrinking convex plane curves.
J. Differentail. Geom. Vol. 23 (1986), 69-96.

\bibitem{Gage-Li-1994} M. E. Gage \& Yi Li,
Evolving plane curves by curvature in relative geometries II.
Duke Math. J. No. 1, Vol. 75 (1994), 79-98.

\bibitem{Gao-Wang-2008} L.-Y. Gao \& Yi-L. Wang,
Deforming convex curves with fixed elastic energy.
J. Math. Anal. Appl. Vol. 427 (2015), 817-829.

\bibitem{Gao-Zhang-2017} L.-Y. Gao \& Y.-T. Zhang,
Evolving convex surfaces to constant width ones.
International Journal of Mathematics No. 11, Vol. 28(2017), 1750082 (18 pages).


\bibitem{Gao-Zhang-2019} L.-Y. Gao \& Y.-T. Zhang,
On Yau's problem of evolving one curve to another: convex case.
J. Differential Equations Vol. 266 (2019), 179-201.




\bibitem{Grayson-1987} M. Grayson,
The heat equation shrinks embedded plane curve to round points.
J. Differential Geom. Vol. 26 (1987), 285-314.

\bibitem{Grayson-1989} M. Grayson,
Shortening Embedded Curves.
The Annals of Mathematics, Second Series No.1, Vol.129 (1989), 71-111.

\bibitem{Hadamard-1898} J. Hadamard,
Les surfaces \`{a} courbures oppos\'{e}es et leurs linges g\'{e}od\'{e}siquese.
J. Math. Pures Appl. 4 (1898), 27-73.

\bibitem{Hopf-1983} H. Hopf,
Differential geometry in the large. Notes taken by Peter Lax and John Gray. With a preface by S. S. Chern.
Lecture Notes in Mathematics, 1000. Springer-Verlag, Berlin, 1983. vii+184 pp.

\bibitem{Lin-Tsai-2009} Y.-C. Lin \& D.-H. Tsai,
Evolving a convex closed curve to another one via a length-preserving linear flow.
J. Differential Equations Vol. 247 (2009), 2620-2636.

\bibitem{Pan-Yang-2019} S.-L. Pan \& Y.-L. Yang,
An anisotropic area-preserving flow for convex plane convex.
J. Differential Equations No. 6, Vol. 266 (2009), 3764-3786.

\bibitem{Singer-2008} D. A. Singer,
Lectures on elastic curves and rods,
AIP Conf. Proc. Vol. 1002 (2008), 3-32.

\bibitem{Tsai-2005} D.-H. Tsai,
Asymptotic closeness to limiting shapes for expanding embedded plane curves.
Invent. Math. Vol. 162 (2005), 473-492,

\bibitem{Tsai-2018} D.-H. Tsai,
On flows that preserve parallel curves and their formation of singularities.
J. Evol. Equ. No. 2, Vol. 18 (2018), 303-321.

\bibitem{Wang-Li-Chao-2017} X.-L. Wang, H.-L. Li, X.-Li Chao,
Length-preserving evolution of immersed closed curves and the isoperimetric inequality.
Pacific J. Math. 290 (2017), no. 2, 467-479.

\bibitem{Wang-Wo-Yang-2018} X.-L. Wang, W.-F. Wo, M. Yang,
Evolution of non-simple closed curves in the area-preserving curvature flow.
Proc. Roy. Soc. Edinburgh Sect. A 148 (2018), no. 3, 659-668.

\bibitem{Whitney-1937} H. Whitney, On regular closed curves in the plane.
Compositio Mathematica Tome 4 (1937), 276-284.

\end{thebibliography}
\end{document}